%% file: 20191103-donley-partitions.tex
\documentclass[graybox]{svmult}

\usepackage{mathptmx}     
\usepackage{helvet}   
\usepackage{courier}
\usepackage{amsfonts} 
\usepackage{amssymb}
\usepackage{makeidx}         % allows index generation
\usepackage{graphicx}        % standard LaTeX graphics tool
\usepackage{multicol}        % used for the two-column index
\usepackage[bottom]{footmisc}% places footnotes at page bottom

\makeindex             % used for the subject index
\begin{document}

\title*{Partitions for semi-magic squares of size three}
% Use \titlerunning{Short Title} for an abbreviated version of
% your contribution title if the original one is too long
\author{Robert W. Donley, Jr.}
% Use \authorrunning{Short Title} for an abbreviated version of
% your contribution title if the original one is too long
\institute{Robert W. Donley, Jr., \at Queensborough Community College, Bayside, New York, \email{RDonley@qcc.cuny.edu}}
%
% Use the package "url.sty" to avoid
% problems with special characters
% used in your e-mail or web address
%
\maketitle

\abstract*{n the theory of Clebsch-Gordan coefficients, one may recognize the domain space as the set of weakly semi-magic squares of size three.  Two partitions on this set are considered: a triangle-hexagon model based on top lines, and  one based on the  orbits under a finite group action. In addition to giving another proof of McMahon's formula, we give a generating function that counts the  so-called trivial zeros of Clebsch-Gordan coefficients and its associated quasi-polynomial.}

\abstract{In the theory of Clebsch-Gordan coefficients, one may recognize the domain space as the set of weakly semi-magic squares of size three.  Two partitions on this set are considered: a triangle-hexagon model based on top lines, and  one based on the  orbits under a finite group action. In addition to giving another proof of McMahon's formula, we give a generating function that counts the  so-called trivial zeros of Clebsch-Gordan coefficients and its associated quasi-polynomial.}

\keywords{Clebsch-Gordan coefficient, semi-magic square, stochastic matrix}

\section{Introduction}
\label{sec:1}

 A basic open problem of elementary representation theory concerns the classification of zeros for Clebsch-Gordan coefficients of $SU(2).$  Regge made the remarkable observation that  the domain space for Clebsch-Gordan coefficients corresponds precisely to weakly semi-magic of squares of size three and that, suitably normalized, the associated group of symmetries of the determinant  acts upon values of Clebsch-Gordan coefficients by sign changes.   In particular, these symmetries preserve the zero locus for Clebsch-Gordan coefficients.

In the notation of \cite{Do}, a  parametrization for the domain of $c_{m, n, k}(i, j)$ is given by the set of semi-magic squares
\begin{equation}\left|\left|\begin{array}{ccc} \phantom{-}n-k & \phantom{-}m-k & \phantom{-}k \\ \phantom{-}i &\phantom{-} j &\phantom{-} m+n-i-j-k\ \\ \phantom{-}m-i & \phantom{-}n-j & \phantom{-}i+j-k\end{array}\right|\right|\end{equation}  with non-negative integer entries. In particular, the Clebsch-Gordan coefficients for a fixed top line record all data for the projection
\[V(m)\otimes V(n)\rightarrow V(m+n-2k). \]
The bottom rows sum according to the projection and may be considered as exponent pairs.  With respect to the natural invariant bilinear form,  the weight vectors
\[ f^i\phi_m\quad{\rm and}\quad f^{m-i}\phi_m\]
 pair non-trivially.
 
The classification of  zeros for Clebsch-Gordan coefficients typically starts by extracting the so-called trivial zeros; these zeros arise from certain fixed points of Regge symmetries, and, as we will show, display a measurable regularity.  The remaining zeros may be further stratified by magic number $J$, the minimal entry of the magic square (the degree), and the order; see \cite{RRV} for a definition of the latter item and for results on non-trivial zeros with $J\le 900$ and beyond.
 
 Our  strategy for enumerating non-trivial zeros is as follows:
 \begin{itemize}
 \item for a fixed $J$, elementary MAPLE programming creates the list of all zeros as magic squares.  Output consists of seven integers - the square (five parameters), the minimal entry in the square, and the determinant of the square,
 \item for odd $J$, there are trivial zeros; the associated magic squares have determinant zero, but non-trivial zeros may also have this property,
 \item such non-trivial zeros are detected by comparing counts for determinant zero with the generating function that counts trivial zeros; these zeros are sparse and readily found, and 
 \item to both enumerate orbits and count orbit sizes, representatives for non-trivial zeros are narrowed down by sorting the list in EXCEL using determinant and smallest value.
 \end{itemize}

In this work, we consider results on enumerating trivial zeros as a proof of concept for techniques associated to two partitions.  Sections 3 through 5 develop a triangle-hexagon model; another proof of McMahon's enumeration of semi-magic squares of size three is given in Section 3 (Proposition 3), and Sections 4 and 5 provide further examples. Sections 6 though 8 consider orbit-based partitioning. The main results are the enumeration of reduced squares in Table 1, the generating function that counts trivial zeros in Section 7 (Theorem 1), and the corresponding quasi-polynomial for trivial zeros in Section 8 (Theorem 2).  Section 9 gives a detailed application of the theory to the example of $J=15.$

\section{Magic squares and Clebsch-Gordan coefficients}
\label{sec:2}

\begin{definition}
Let $\mathbb{M}_3$ denote the monoid of weakly semi-magic squares of size three; that is, each element of $\mathbb{M}_3$  is of the form
\begin{equation}
M=\left|\left|
\begin{array}{ccc}
 \, a &\ b &\  k\, \\
\,  r & \ * &\ *\, \\
\,  * & \ * & \ c\, \\
\end{array}\right|\right|
\end{equation}
such that 
\begin{itemize}
\item all entries are nonnegative integers, and
\item all line sums along rows and columns are equal.

\end{itemize}
\item This sum, the {\bf magic number}, equals $J=a+b+k.$ Such matrices are also called {\bf integer doubly-stochastic matrices}.
\end{definition}

\begin{definition}
Fix $J\ge 0.$  Let $\mathbb{M}_3(J)$ denote the subset in $\mathbb{M}_3$ of all semi-magic squares with magic number $J.$  Define $H_3(J)$ to be the cardinality of $\mathbb{M}_3(J).$
\end{definition}

Of course, when $J>0,$ any $M$ in $\mathbb{M}_3(J)$ becomes a proper doubly-stochastic matrix with rational entries upon division by $J.$  The earliest formula for $H_3(J)$ is due to MacMahon \cite{PMM}; see \cite{St2}, and Proposition 3 below.  For a general overview of properties of $H_n(J)$, see Ch.1 of \cite{St1}.  An extensive presentation of the theory of doubly-stochastic matrices with respect to composite quantum systems may be found in \cite{Lo2}.

With this parametrization, the corresponding Clebsch-Gordan coefficients belong to the tensor product
\[ V(b+k)\otimes V(a+k)\rightarrow V(a+b);\]
each highest weight is the pairwise sum of elements in the top line.

\begin{definition}  The 
Clebsch-Gordan function
\[
C:  \mathbb{M}_3\rightarrow \mathbb{Z}
\]
is the map 
$M=
\left|\left|\begin{array}{ccc}
 a & \ b &\  k\\
 r & \ s &\  *\\
\,  * & \ * &\  c\\
\end{array}\right|\right| \mapsto$
\[
C(M)=\sum\limits_{l=0}^k (-1)^l  \left(\begin{array}{c} c \\ r-l \end{array}\right)  \left(\begin{array}{c} b+k-l\\ b\end{array}\right) \left(\begin{array}{c} a+l\\ a\end{array}\right).
\]
\end{definition}

Here $s=c-r+k,$ and
$C(M)$ is the coefficient of 
$$x^{s}y^r$$
 in the power series expansion of 
\[\frac{(x+y)^c}{(1-x)^{b+1}(1+y)^{a+1}}.\]

\section{A partition for semi-magic squares of size three}
\label{sec:3}

For square matrices of size three, let $G$ be the group of determinantal symmetries;  that is,
\begin{itemize}
\item  $G$ is generated by row switches, column switches, and transpose,
\item  every element $g$ of $G$ may be expressed uniquely as  
\begin{equation}g=R(\sigma)C(\tau)T^\epsilon\quad {\rm with}\quad \sigma, \tau \in S_3,\ \epsilon\in\{0, 1\}\end{equation}and
\item  $|G|=72.$
\end{itemize}

In the above factorization, $R(\sigma)$ denotes a permutation of rows, $C(\tau)$ a permutation of columns, and $T$ the transpose operation. For basic group theoretic properties of these symmetries, such as conjugacy classes, normal subgroups, and character table, see \cite{RVR}.

These symmetries preserve 
\begin{enumerate}
\item the semi-magic square property, 
\item the magic number $J$ and thus each ${M}_3(J)$, and 
\item the zero locus for Clebsch-Gordan coefficients. 
\end{enumerate}

The classification of zeros of Clebsch-Gordan coefficients has traditionally been a subject of intense study, largely focused in the mathematical physics literature, and remains an open problem.  In the present work and preceding issues (\cite{Do}, \cite{Do2}, \cite{Do3}), one goal has been to reconstruct parts of the general theory using combinatorial methods.   A goal of this work is to give a natural visual representation of the domain space. That is, we  parametrize this subset of semi-magic squares as
\begin{equation}
\left|\left|\begin{array}{ccc}
\,  a &\ b &\  k\ \\
\, r & \ * & \ *\ \\
\, * & \ * & \ c\ \\
\end{array}\right|\right|  \subset \mathbb{N}^5
\end{equation}
in terms of a triangle-hexagon model;  generally, each lattice point in a three-dimensional triangular cone represents a hexagonal family of magic squares.

The partitioning of this set occurs in three steps: 

\begin{itemize}
\item  fix a magic number $J \ge 0,$
\item  organize the compositions of $J$ with three parts into an equilateral triangle of size $J+1$, and 
\item  organize all magic squares with a fixed top line, as a composition of $J,$ into a possibly degenerate hexagon.
\end{itemize}
In this way, the data attached to $$V(b+k)\otimes V(a+k) \rightarrow V(a+b)$$ is indexed by all magic squares with top line
 $(a, b, k).$

To organize the top lines for a fixed $J\ge 0,$ we index the first two entries of the triple as usual for matrices, but starting at index zero.  These triples may then be arranged as an equilateral triangle as noted.  For any top line in the triangle, the entries denote the number of rows above, the number of spaces to the left, and the number of spaces to the right, respectively,

To further analyze the set of magic squares associated to a top line, we consider $(r, c)$ as coordinates in a square matrix of size $J+1$ in which the first row and column  are indexed by zero.  By the magic square property, $r$ and $c$ are subject to the inequalities
\begin{equation}0\le r \le b+k,\qquad 0\le c \le a+b,\qquad -k\le c-r \le a.\end{equation}
The solution set of these inequalities may be interpreted as a rectangle with two congruent right triangles removed from the upper right and lower left corners.   Since $(r, c)=(0, 0)$ is always a solution of these inequalities, the solution set always has one vertex in the upper left corner.  

Furthermore the shape of the solution set is determined by the following  side lengths:
\begin{itemize}
\item vertical side:  $k+1,$
\item diagonal side: $b+1,$
\item horizontal side: $a+1.$
\end{itemize}
If exactly one of $a, b,$ or $k$ is zero, the solution set is a parallelogram, and  we obtain a line segment if two parameters equal zero; otherwise, the solution set is a hexagon with parallel opposing sides. Proposition 2 below follows directly.

Considering the characterization by side lengths, we see that  a  permutation of columns permutes the entries of a given top line, with the effect that the corresponding hexagon has the same shape in a new orientation.
 
Finally, we note formulas for enumerating the various objects of interest:
\begin{proposition} The number of top lines associated to a fixed $J$ is given by
\begin{equation}1+2+3+ \dots + (J+1) = \frac{(J+1)(J+2)}2.\end{equation}
\end{proposition}

\begin{proposition} Let $|(a,b,k)|$ denote the number of semi-magic squares with top line $(a, b, k)$.  Then 
\begin{equation} |(a, b, k)| =1 + (a + b + k) + (ab + ak + bk).
\end{equation}
\end{proposition}
In particular, $|(a, b, k)|$ is invariant under permutations of $a,$ $b$ and $k$, a fact consistent with our characterization of column permutations as isometries of the corresponding hexagon.  See Sect. 5 for an algorithm that computes the array of all $|(a, b, k)|$ for a given $J$.

\begin{proposition}[\cite{Slo}, A002817] Let $H_3(J)$ be the number of semi-magic squares with magic number $J.$ Then 
\begin{equation}
H_3(J)  =\frac{(J+1)(J+2)(J^2+3J+4)}8.
\end{equation}
\end{proposition}

\begin{remark}  For alternative expressions to (8), we have
\begin{eqnarray}
H_3(J)&= &3\left(\begin{array}{c} J+3\\ 4 \end{array}\right) + \left(\begin{array}{c} J+2\\ 2 \end{array}\right)\\
  &= & \left(\begin{array}{c} J+5\\ 5 \end{array}\right) - \left(\begin{array}{c} J+2\\ 5 \end{array}\right)\\
&= &\left(\begin{array}{c} J+4\\ 4 \end{array}\right) + \left(\begin{array}{c} J+3\\ 4 \end{array}\right) + \left(\begin{array}{c} J+2\\ 4 \end{array}\right).
\end{eqnarray}
Equation (9) is given in \cite{PMM},  Art. 407. In \cite{St2}, one finds (10) and (11) in an exercise (Ch. 2, Ex. 15, pp. 225, 236);  see also \cite{Bo}.
From the third equality, we have the generating function 
$$\sum\limits_{J\ge 0} H_3(J) x^J = \frac{1+x+x^2}{(1-x)^5},$$
also found in Sect. 4.6.1 of \cite{St2}.  For general properties of $H_n(J)$, see \cite{St1}, Ch. 1.1-1.9.
\end{remark}

\begin{proof}[Proposition 3]   From the partition triangle, we have that
\begin{equation}H_3(J)=\sum\limits_{k=0}^J\sum\limits_{b=0}^{J-k} |(J-b-k, b, k)|.\end{equation}
Using Proposition 2 and the identities
\[\sum\limits_{k=0}^N k = \frac{N(N+1)}{2},\quad \sum\limits_{k=0}^N k^2 = \frac{N(N+1)(2N+1)}{6},\quad \sum\limits_{k=0}^N k^3 = \frac{N^2(N+1)^2}{4},\]
the result follows from direct computation. Alternatively, the identities imply that $H_3(J)$ is a polynomial in $J$ of degree 4, and the coefficients may be simply computed using the first five values of $H_3(J):$
\[H_3(0)=1,\ \ H_3(1)=6,\ \ H_3(2)=21,\ \ H_3(3)=55,\ H_3(4)=120.\quad\qed \]
\end{proof}

\begin{remark}
An outline for this proof is noted in Sect. 6.2 of \cite{Lo2};   Proposition 2 provides the missing step in (6.13). 
\end{remark}

\section{Examples}
\label{sec:4}

We consider  $\mathbb{M}_3(4)$, the subset of semi-magic squares with $J=4.$  The triangle of top lines and corresponding magic square counts are given as Fig. 1. Here we have boxed entries to be used below. 
\begin{figure}
\includegraphics[scale=.9]{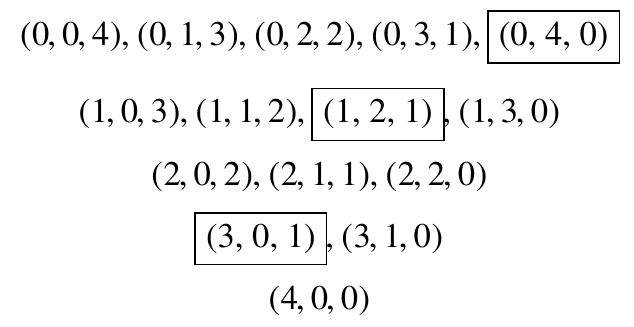}
\hspace{20pt}
\includegraphics[scale=1]{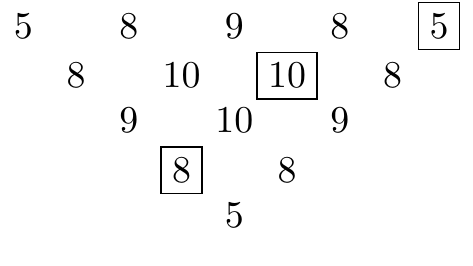}
\caption{\ $J=4$: {\bf a} Top lines and {\bf b} magic square counts.}
\label{Figure 1}
\end{figure}

Applying MacMahon's formula, we obtain 120 magic squares, confirmed by the counts above, as are the 15 top lines.

Now consider the boxed top line $(0, 4, 0)$ with 5 magic squares:

\[\left|\left|\begin{array}{ccc}\ 0 \ & 4 \ &  0\ \\ r & 0\ &   * \\ \, * & 0\ & c \end{array}\right|\right|\mapsto \left[\begin{array}{ccccc} *\ &  & & & \\ & * \ &  & & \\ & & * & & \\ & & & \fbox{$*$} &\\ & & & & *\end{array}\right].\]
The starred entries on the right correspond to $(r, c)$ that yield a valid magic square; the boxed entry corresponds to $(r, c)=(3, 3),$ which gives the magic square   
\[\left|\left|\begin{array}{ccc} \ 0 \ &  4 \ &  0\ \\  3  & 0 \  &  1 \\ 1 &  0\  & 3 \end{array}\right|\right|.\]
On the other hand, the pair $(r, c)=(1, 0)$ does not produce a valid magic square, signified by the blank space below the first star.

Next we have the boxed top line $(3, 0, 1)$ with 8 magic squares:

\[\left|\left|\begin{array}{ccc}\ 3 \ & 0 \ &  1\ \\ r & *\ &   *\ \\ \, * & * & c\ \end{array}\right|\right|\mapsto 
\left[\begin{array}{cccc}  * \ & * &  * & *  \\ 
\,  * \  & \, *  &  \fbox{$*$} & \, *  \\   
\phantom{-} & \phantom{-}  & \phantom{-}  & \phantom{-}  \\ 
\phantom{-} & \phantom{-} & \phantom{-} & \phantom{-} 
\end{array}\right].\]
The boxed entry corresponds to $(r, c)=(1, 2),$ corresponding to the magic square   
\[\left|\left|\begin{array}{ccc} \ 3 \ &  0 \ & \ 1\ \\  1  & 2   &  1 \\ 0 &  2  & 2 \end{array}\right|\right|.\]

Finally consider the boxed top line $(1, 2, 1)$ with 10 magic squares:

\[\left|\left|\begin{array}{ccc}\ 1 \ & 2 \ &  1\ \\ r & *\ &   *\ \\ \, * & *\ & c\ \end{array}\right|\right|\mapsto \left[\begin{array}{cccc} \, *\ & * &   &   \\  \, * &  * &  * &  \  \\   &  * & \fbox{$*$} & * \ \\ & & * & * \  \\ \  \end{array}\right].\]
The boxed entry corresponds to $(r, c)=(2, 2),$ giving the magic square   
\[\left|\left|\begin{array}{ccc} \ 1 \ &  2 \ &  1\ \\  2  & 1 \  &  1\ \\ 1 &  1\  & 2\ \end{array}\right|\right|.\]

Of course, if we evaluate $C(M)$ for all $M$ in each array, we obtain, respectively,  
\[ \left[\begin{array}{ccccc} 1 \ &  & & & \\ & 1 \ &  & & \\ & & 1\ & & \\ & & & 1\ &\\ & & & & 1\ \end{array}\right],\quad  \left[\begin{array}{cccc} \ \ 1  &\ \  1 & \ \  1 & \ \ 1  \\ -4 &  -3 & -2 & -1   \end{array}\right], \quad
\left[ \begin{array}{cccc}\ \, 3 & \ \ 3 &   &   \\ -2 \  & \ \ 1 & \ \ 4 &  \ \\   &  -2\, & -1 & \ \ 3 \, \\ & & -2 \, & -3  \  \end{array}\right].\]

\section{Examples of zeros}
\label{sec:5}

We now apply the partition of  Sect. 3 to organize the zeros of $C(M)$ in each $\mathbb{M}_3(J).$  For instance, if $J=8,$  Fig. 2 gives an enumeration by top lines of magic squares and zeros of $C(M).$

\begin{figure}
\includegraphics[scale=.7]{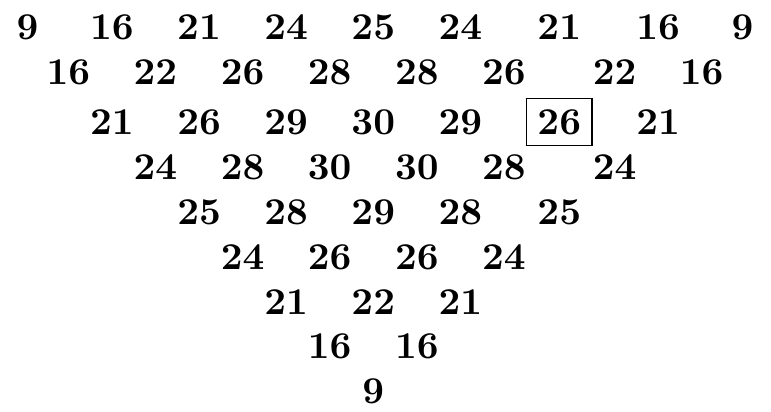}
\hspace{30pt}
\includegraphics[scale=.7]{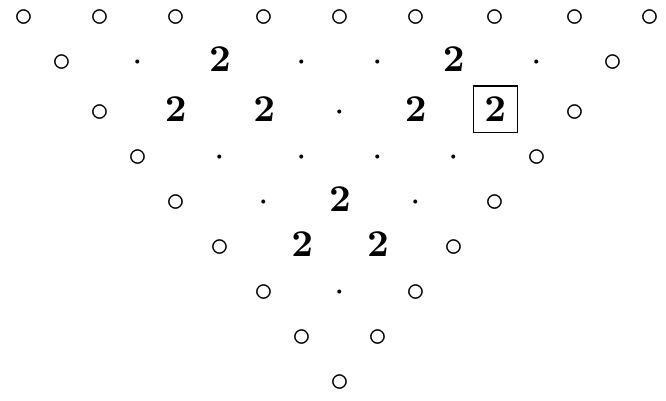}
\caption{\ $J=8$: {\bf a} Magic square counts and  {\bf b} zero counts.}
\label{Figure 2}
\end{figure}
For the figure on the left, we note that the magic square counts may be generated algorithmically using (7). That is,
\begin{itemize}
\item the vertices have value $J+1,$
\item the differences along rows progress arithmetically, decreasing by increments of 2, and
\item the starting differences for each row begin with $J-1$ in the top row and decrease by increments of 1.
\end{itemize}
One notes that values are preserved under the symmetries induced by permutations of columns.  Top line orbits have orders 1, 3, or 6.   Any fixed point is of the form $(a, a, a)$ and corresponds to  the central entry in the triangle when $J=3a.$

On the right in Fig. 2, we see that there are 18 zeros of $C(M)$ for $J=8.$  These zeros correspond to a single orbit of magic squares under the full determinantal group, with representative noted below.

The boxed entries correspond to the top line $(2, 5, 1),$ confirmed by counting rows above and spaces to each side of the box.  The hexagon corresponding to top line $(2, 5, 1)$ is given by
\begin{equation}\left[ \begin{array}{cccccccc} \ \ 6 & \ \ 6 & \ \ 6 & & & & & \\ -3\, & \ \ 3  & \ \ 9 & \ 15 &  & & & \\  & -3\, & \ \ \fbox{$0$} & \ \ \, 9 & \ 24 & & & \\ & &  -3 & -3  & \ \ \, 6 & \ 30 & & \\ & & & -3 & -6 &\ \  \fbox{$0$} & \ \ 30 &  \\ & & &  & -3 & -9 &\ \, -9 & \ \ \, 21 \\ & & &  &  & -3 & -12 & -21\end{array}\right],\end{equation}
from which one readily counts 26 magic squares and two zeros. In particular, for the boxed zeros, we have coordinates  $(r, c) \ =\ (2, 2), \ (4, 5),$ yielding the magic squares
 \[\left|\left|\begin{array}{ccc}\ 2\ & 5\ & 1\ \\ 2 & 1\ & 5\  \\ 4 & 2\ & 2\ \end{array}\right|\right|,\qquad \left|\left|\begin{array}{ccc}\ 2\ & 5\ & 1\ \\ 4 & 2\ & 2\ \\ 2 & 1\ & 5\ \end{array}\right|\right|. \]
 These squares are related by the Weyl group symmetry, which switches the lower rows; the effect on the hexagon is to rotate by 180 degrees, while the Clebsch-Gordan coefficients transform according to formula (8.5) in \cite{Do}.
Additionally, we recall that this hexagon records all data for the tensor product 
\[ V(6)\otimes V(3)\rightarrow V(7).\]

In general, we will consider only the triangle for the zero locus of $C(M)$ in $\mathbb{M}_3(J)$ and triangles for the orbits of zeros under the full symmetry group.  For example, the zero locus corresponding to $J=24$ contains six orbits contributing 252 zeros.  
Fig. 3 gives the full portrait of zero counts, and Fig. 4 gives the portrait of zero counts for the orbit containing the twelve zeros for top line $(8, 8, 8)$.  

The hexagon for top line $(8, 8, 8)$ is preserved under a symmetry group of type $D_{12},$ the dihedral group of order 12, and all zeros for $(8, 8, 8)$ consist of a single orbit under $D_{12}$.  A representative of this orbit is given by

\begin{figure}
\includegraphics[scale=.9]{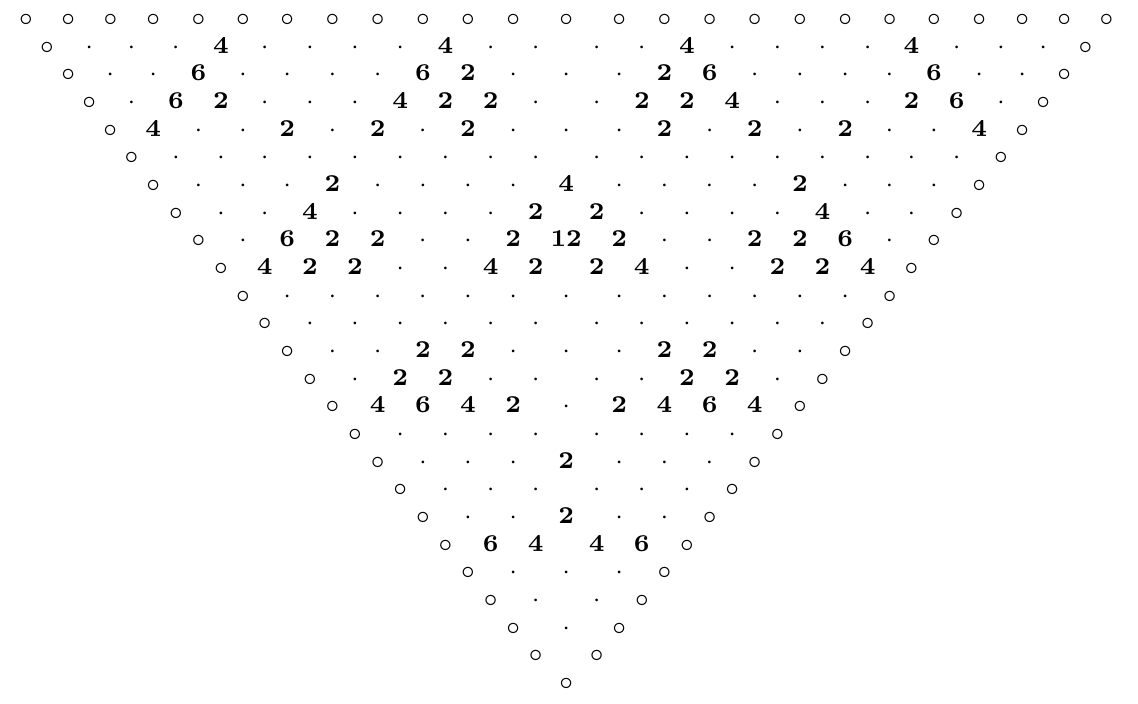}
\caption{The 252 zeros for $J=24$.}
\label{Figure 3}
\end{figure}

\begin{figure}
\includegraphics[scale=.9]{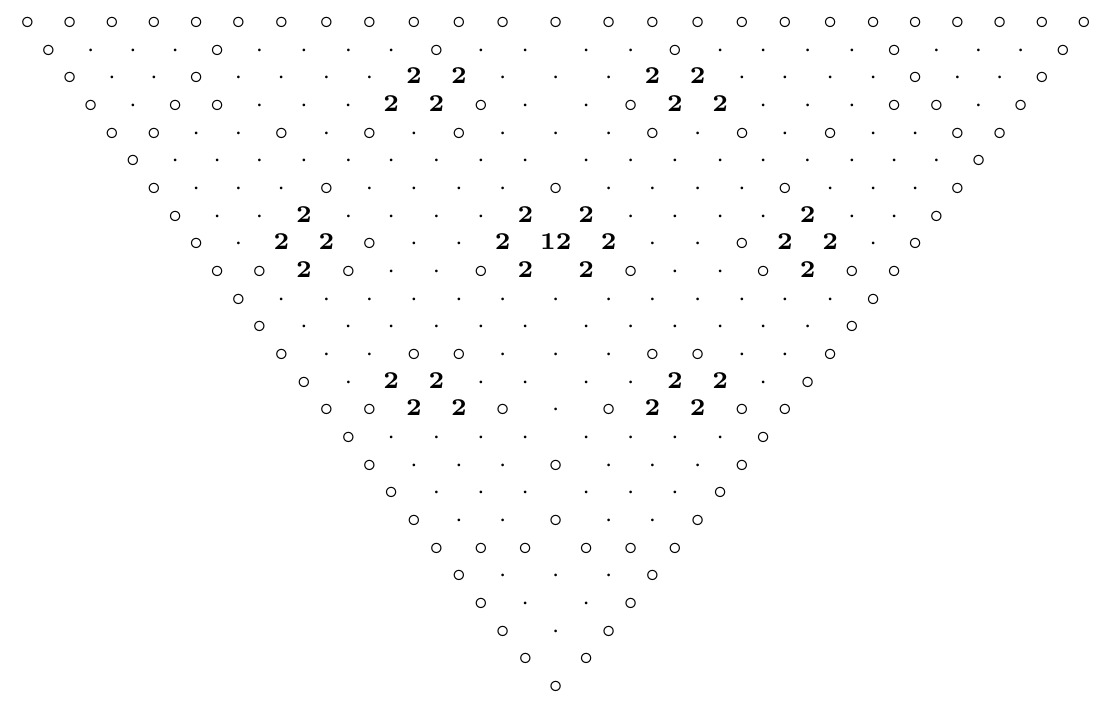}
\caption{The 72 zeros for the orbit of $M_{(8, 8, 8)}$.}
\label{Figure 4}
\end{figure}

\[
M_{(8, 8, 8)}=
\left|\left|\begin{array}{ccc}
 \ \,  8\, &  \ 8 &\ \  8\, \\
 \ 2&  \ 9 & 13\, \\
 14 & \ 7 & \ \ 3\, \\
\end{array}\right|\right|.
\]

A necessary condition to have an orbit of order 12 is that the hexagon be regular, or equivalently that the top line is $(a, a, a),$ the center of the triangle when $J=3a.$ As of this writing, the author knows of no other examples; see Fig. 2 in \cite{Do2}.

\section{Orbits and a second partition}
\label{sec:6}

For another approach to Proposition 3, we consider a different partitioning of each $\mathbb{M}_3(J)$, closer in spirit to Sect. 4 of \cite{BZ}, which uses convex geometry and specific representatives to construct generating functions for related classes of magic squares.  

Consider the following definitions:

\begin{definition} Let $U$ denote the element of $\mathbb{M}_3(3)$ with all entries having value 1.
\end{definition}

\begin{definition}  We say that an element $M$ of $\mathbb{M}_3$ is a {\bf reduced} if $0$ occurs as an entry in $M$.  For a fixed $J\ge 0$, we denote the set of all reduced squares in $\mathbb{M}_3(J)$ by $N_3(J)$.  
\end{definition}

If any entry of a top line equals zero, then the corresponding polygon consists entirely of reduced magic squares; otherwise, the boundary of the corresponding hexagon consists precisely of the reduced magic squares with this top line.  

\begin{definition}  Let $\mathbb{M}_3^s$ denote the subset of $\mathbb{M}_3$ of squares with smallest entry $s$, and let $\mathbb{M}_3^s(J)=\mathbb{M}_3^s \cap \mathbb{M}_3(J)$. 
\end{definition}

The partitioning in question is rooted in the following observation:
\begin{lemma}
Let $M$ be an element of $\mathbb{M}^s_3(J)$.  Then there exists a unique decomposition of $M$ as 
\[M=sU+N\]
with $N$ a reduced square in $\mathbb{M}_3(J-3s).$
\end{lemma}

Of course,
\begin{equation}
\mathbb{M}_3^s(J)=sU+N_3(J-3s).
\end{equation}
When $J\ge 3,$ we may partition $\mathbb{M}_3(J)$ as
\begin{eqnarray}
\mathbb{M}_3(J) &=& N_3(J) \cup (U+\mathbb{M}_3(J-3))\nonumber\\
&=& N_3(J) \cup (U+N_3(J-3)) \cup (2U+N_3(J-6)) \cup\dots.
\end{eqnarray}
The first equation partitions $\mathbb{M}_3(J)$ into reduced and non-reduced squares, while the latter partitions by smallest value.  In any case, the cardinality of $\mathbb{M}_3(J)$ may be obtained from the cardinalities of $N_3(l)$ with $l\le J.$ 

Now $U$ is unchanged under the group of determinantal symmetries, so it is enough to decompose each $N_3(J)$ into orbits under the group $G$ and count orbits using a consistent set of representatives.  Omitting a representative for orbit size 72, the following representatives correspond to the data  in Table 1:

\begin{eqnarray}
& \left|\left|\begin{array}{ccc}
 \ a& \  a\ &\ a\, \\
 \ a & \ a &\ a\\
 \ a & \ a &\ a\\
\end{array}\right|\right|,\ \left|\left|\begin{array}{ccc}
 \  b & \  a\ &\ a\, \\
 \ a & \ b &\ a\\
 \ a & \ a &\ b\\
\end{array}\right|\right|,\ 
\left|\left|\begin{array}{ccc}
 \ b& \  c\ &\ c\, \\
 \ c & \ a &\ a\\
 \ c & \ a &\ a\\
\end{array}\right|\right|,\ 
\left|\left|\begin{array}{ccc}
 \ a& \  c\ &\ b\, \\
 \ c & \ b &\ a\\
 \ b & \ a &\ c\\
\end{array}\right|\right|,\nonumber\\
&\left|\left|\begin{array}{ccc}
 \ b& \  a\ &\ a\, \\
 \ a & \ c &\ d\\
 \ a & \ d &\ c\\
\end{array}\right|\right|,\ 
\left|\left|\begin{array}{ccc}
 \ a& \  d\ &\ f\, \\
 \ d & \ b &\ e\\
 \ f & \ e &\ c\\
\end{array}\right|\right|,\ 
\left|\left|\begin{array}{ccc}
 \ d& \  e\ &\ f\, \\
 \ a & \ b &\ c\\
 \ a & \ b &\ c\\
\end{array}\right|\right|\ 
.
\end{eqnarray}

Stabilizer subgroups $D_{2n}$ of $G$ in Table 1 are dihedral groups with $2n$ elements.

\begin{table}
\caption{Orbit counts for reduced squares}
\label{tab:1}
\begin{tabular}{p{2cm}|p{3cm}p{2.5cm}|p{3.5cm}}
\hline\noalign{\smallskip}
  & \# Orbits of & &  \\
Orbit Size & Reduced Squares & & Stabilizer \\
\noalign{\smallskip}\svhline\noalign{\smallskip}  
1 & 1 & $J=0$ & G\\   \hline
 6 & 1 & $J$ odd  & $\langle R(\sigma)C(\sigma), T\rangle\cong D_{12}$ \\
  & 2 & $J$ even &  \\   \hline
9  & 1 & $J=4t+2$  & $\langle R(23), C(23), T\rangle\cong D_{8}$ \\
    & 2 & $J=4t$  &  \\  \hline
12  & $t-1$ & $J=2t$  & $\langle R(123)C(132), T\rangle\cong S_3$ \\   
    & $t$ & $J=2t+1$  &  \\    \hline
18  & $2t$ & $J=2t+1$  & $\langle R(23)C(23), T\rangle$ \\   
      & $5t-3$ & $J=4t$  &  $\qquad\qquad\cong \mathbb{Z}/2\times \mathbb{Z}/2$\\   
      & $5t$ & $J=4t+2$  &  \\    \hline
 36  & $(t-1)(3t-4)/2$ & $J=2t$  & $\langle T\rangle\cong \mathbb{Z}/2$ \\    
       & $3t(t-1)/2$ & $J=2t+1$  &  \\    \hline
 36  & $t-1$ & $J=4t$  & $\langle R(23)\rangle\cong \mathbb{Z}/2$ \\     
       & $t$ & $J=4t+2$  & \\      \hline
 72  & $t(t-1)(t-2)/6$ & $J=2t+1\ge 7$  & $\langle e\rangle$  \\    
       & $t(t-1)(4t-5)/3$ & $J=4t\ge 8$  \\  
       & $t(t-1)(4t+1)/3$ & $J=4t+2\ge 10$   \\  
  \noalign{\smallskip}\hline\noalign{\smallskip}
\end{tabular}
\end{table}

To verify Table 1, the counts for orbits of constant or linear size may be computed directly.  For the orbit type of size 36 with symmetric representative, we note that two cases occur: either two zeros may be placed with $f=0$ or a single zero with $a=0.$ Then $M$ is the determined by the choice of $d$ and $e$.  The  range of allowed values in either case forms one or two  isosceles right triangles.  Counting over arithmetic progressions, we obtain a quadratic polynomial in $t.$  Initial conditions determine the coefficients.

For the orbits of size 72, we note that
\begin{equation}|N_3(J)|=H_3(J)-H_3(J-3)=\frac{3J(J^2+3)}2;\end{equation}
the formulas follow by  removing counts for orbits of smaller size in each case.

Of course, if we calculate the latter orbit counts directly, one could give another proof of Proposition 3.  Instead, we  illustrate the method for this proof to count trivial zeros in $\mathbb{M}_3(J)$ in the next section.

\section{Trivial zeros}
\label{sec:7}

In this section, we define the class of trivial zeros of $C(M)$ and give generating functions for both the number of orbits and number of trivial zeros in $\mathbb{M}_3(J)$ for any $J$. In this work, we apply the notion of trivial only to elements in $\mathbb{M}_3(J).$ Traditionally, the domain for Clebsch-Gordan coefficients is extended by zero in an appropriate manner.

\begin{definition}
Suppose $M$ is in $\mathbb{M}_3(J)$ with all positive entries. Then $M$ is called a {\bf trivial zero} of $C(M)$ if and only if $J$ is odd and $M$ has at least one  pair of matching rows or columns.  
\end{definition}

With the matching condition and even $J$, $C(M)$ is non-zero and may be computed explicitly. An orbit of zeros for $C(M)$ under $G$ corresponding to a trivial zero consists of trivial zeros. We call such an orbit a {\bf trivial orbit} of zeros.

\begin{definition}
For $J\ge 0$, let  $O^T_3(J)$ be the number of trivial orbits in $\mathbb{M}_3(J)$, and let $H_3^T(J)$ be the number of trivial zeros in $\mathbb{M}_3(J)$. 
\end{definition}

\begin{theorem}
The generating function that counts the number of trivial orbits $O_3^T(J)$ in $\mathbb{M}_3(J)$ is given by 
\begin{equation} \sum\limits_{J\ge 0}\ O_3^T(J)\ x^J = \frac{x^3(1+2x^6-x^{10})}{(1-x^4)^2(1-x^6)}.
\end{equation}

The generating function that counts the number of trivial zeros $H_3^T(J)$ in $\mathbb{M}_3(J)$ is given by 
\begin{equation} \sum\limits_{J\ge 0}\ H_3^T(J)\ x^J = \frac{x^3(1+9x^2+16x^4+27x^6+19x^{8})}{(1-x^4)^2(1-x^6)}.
\end{equation}
\end{theorem}

\begin{proof}
To count orbits, we first consider those orbits in Table 1 corresponding to semi-magic squares with matching rows or columns.  Only three types contribute; these orbits have sizes 1, 9, or 36.   

For example, consider the number of trivial orbits for stabilizer of type $\langle R(23)\rangle$ in $\mathbb{M}_3(J)$ with $J$ odd.  Reduced squares of this type only occur for even $J$;  thus none contribute directly to $O^T_3(J).$  

Suppose such a reduced square $M$ occurs in $\mathbb{M}_3(J')$ with $J'=4t$. Addition of $U$ to $M$ once gives a trivial orbit for $J=4t+3$; contributions to odd $J$ follow now by repeatedly adding $2U$.  To the generating function that counts trivial orbits, the contribution for reduced squares of this type in $\mathbb{M}_3(4t)$  is
\begin{equation}(t-1)x^{4t+3}(1+x^6+x^{12}+\dots) = \frac{(t-1)x^{4t+3}}{1-x^6}.\end{equation}
Summing over positive $t$ and noting
\[1+2x+3x^2+\dots = \frac{1}{(1-x)^2},
\]
we obtain
\[\frac{x^{11}}{(1-x^6)(1-x^4)^2}.\]
Likewise, for reduced squares associated to $J'=4t+2,$ the contribution is
\[\frac{x^9}{(1-x^6)(1-x^4)^2}. \]

Thus the generating  function that counts trivial orbits of this type is
\begin{equation}\frac{x^9+x^{11}}{(1-x^6)(1-x^4)^2}.\end{equation}

A similar calculation shows that the orbit counts for sizes 1 and 9 are given by
\begin{equation}\frac{x^3}{1-x^6},\quad \frac{x^5+2x^7}{(1-x^4)(1-x^6)},\end{equation}
respectively.  Both parts of the theorem now follow. $\qed$
\end{proof}

\section{Triangles  for trivial zeros}
\label{sec:8}

Given a fixed $J$, the top line partition for trivial zeros admit a uniform description, as do the orbits under $G$.  There are three components to the top line partition:
\begin{itemize}
\item in general, for each top line $(a, b, k)$ with all odd entries,  values of 1 occur in alternating rows and then alternate along these rows,
\item when $J\ne 3a$, the bisectors from each vertex to the opposite edge have values that progress from 1 to $\frac{J-1}2$; if $J=3a,$ a similar progression occurs, but the value of the middle entry is replaced by  $J-2$, and
\item the triangle with vertices at edge midpoints consists of values of 2, except with values of 3 occurring from the first item and values on the bisectors given by the second item.
\end{itemize}

For example, consider the triangle of zeros for $J=13,$ all of which are trivial:
\begin{figure}
\includegraphics[scale=1]{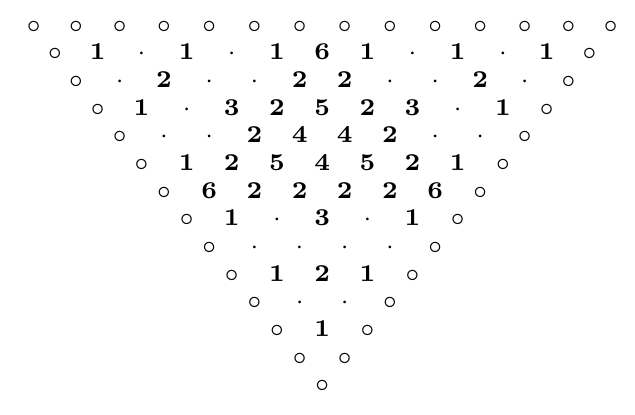}
\caption{\ All zeros for $J=13$ are trivial.}
\label{Figure 5}
\end{figure}

To explain these counts, it may also help to consult the next section for the example of $J=15$. First, the pattern for value 1 is Proposition 15 of \cite{Do2};  the Weyl group symmetry on a hexagon can only have a fixed point at the center, and this center exists and has $C(M)=0$ exactly when all entries of the top line are odd and the lower rows of $M$ match.  For other trivial zeros, any pair of columns or one of the other row pairs match, and these zeros occur as doublets under the Weyl group symmetry. 

For the bisector progressions, see Propositions 11, 13, and 14 of Sect. 4 in \cite{Do2};  in these cases,  the top line has repeat values, and the hexagon has additional symmetries.  When the top line has a matching pair and distinct third value, the hexagon has four sides of equal length and a diagonal of trivial zeros occur; the values along a bisector change with the width of the hexagon. For top line $(a, a, a)$, the hexagon is regular, and three diagonals of zeros of equal length intersect at the center.

Finally, for the central triangle with midpoint vertices, the values of 2 arise when the upper row of $M$ matches only one of the other rows. Off the bisectors, values of 3 in the triangle occur when the top line has all odd entries.

With this enumeration of trivial zeros, we may directly determine $H_3^T(J)$ as a quasi-polynomial, implied by Theorem 1; that is, $H_3^T(J)$ is a polynomial on each remainder class modulo 12.

\begin{theorem}  Let $H_3^T(J)$ denote the number of trivial zeros of $C(M)$ with magic number $J$. For $J\ge 0,$  
\[
H^T_3(J)=\quad
 \begin{array}{l@{\qquad}l@{\qquad}l} 
 0 & J\ even\\[1mm]
\frac34 (J-1)(J-2) &  J=1, 5\ (mod\ 12)\\[1mm]
\frac34 (J-1)(J-2) +4 & J=9\ (mod\ 12)\\[1mm]
\frac34 (J+1)(J-4) & J=7, 11\ (mod\ 12)\\[1mm]
\frac34 (J+1)(J-4) +4 & J=3\ (mod\ 12).
\end{array}
\]
\end{theorem}

\begin{proof}  Summing over arithmetic progressions, the contributions of 1 and 2 from the first and third bullet points are given by 
\[ \frac{J^2-1}8\quad{\rm and}\quad  \frac{(J+1)(J+3)}4,\]
 respectively.  The bisector counts are given by
 \[  \frac{3(J^2-1)}8 \]
 unless $J=3a$, in which case this count is off by a linear error.  The total bisector error in the full count, which is dependent on the parity of the bisector length and multiplicity 3, is linear with quasi-period 12.  Thus, for each odd residue class, the leading term is $\frac34 J^2,$ and values for $1\le J< 24$ determine the remaining coefficients.
$\qed$
\end{proof}

We may further decompose the triangle of trivial zeros into trivial orbits.   Again see the next section for the complete case when $J=15.$  In this case, non-zero entries in the triangle may take values 1, 2, or 6.   A sextet of trivial zeros occur at the midpoint $(a, a, a)$ only if $J=3a$; these six zeros are equidistant from the center along the three diagonals in the corresponding hexagon.  Doublets of trivial zeros occur in two cases.  If the zero doublet occurs on a bisector,  a representative  has matching columns; otherwise, another row must match to the first row.

In general, for a trivial orbit of size 1, we must have $J=3a,$ and the triangle for this unique orbit has a single 1 in the middle entry.   For orbit size 9 (resp. $36$), the convex hull of the non-zero entries in the triangle form an equilateral triangle (resp. hexagon) with a 1 at each vertex and a 2 at the midpoint of each edge.  

The hexagon case is filled out with three parallelograms based at the midpoints of the edges and inside the triangle formed by these midpoints of the long edges. The entries adjacent to the midpoints are given by intersecting the equilateral triangles formed from the midpoints; these entries always have value 2. The remaining vertices may meet at the center with value 6 only if $J=3a;$ otherwise, these vertices have value 2.

\section{Example:  Triangles for $J=15$}
\label{sec:9}

To illustrate the various portraits described in the previous section, the case of $J=15$ is large enough to exhibit most of the phenomena described in the previous section. Here we have eight orbits, seven trivial and one non-trivial, and trivial orbits of size 36 that both contain and do not contain the center.   Fig. 6 displays counts for all zeros and all trivial zeros.  Figures 7 through 9 give the trivial orbits, and Fig. 10 shows the only non-trivial orbit.   We omit the triangle for the trivial orbit with top line $(5, 5, 5),$ which consists of a single 1 at the center of the triangle.

Finally, we recall the 36-pointed star of Theorem 2 in \cite{Do2}. In the case of $J=15,$ this star is composed of 73 trivial zeros, represented as the union of trivial orbits corresponding to top lines $(1, 5, 9)$, $(3, 5, 7),$ and $(5, 5, 5).$ See Fig. 9 below.

\begin{figure}
\includegraphics[scale=.82]{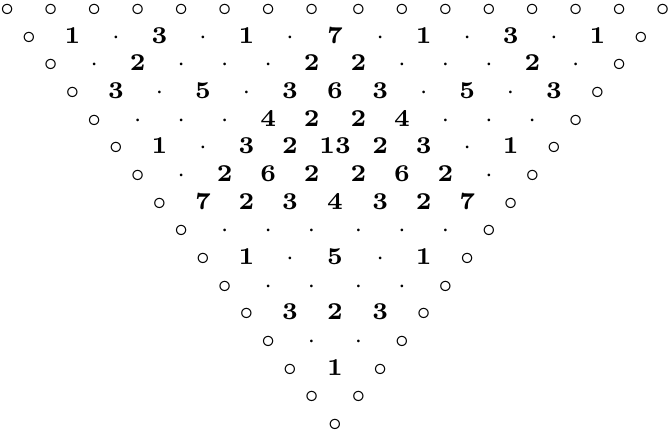}
\hspace{7pt}
\includegraphics[scale=.82]{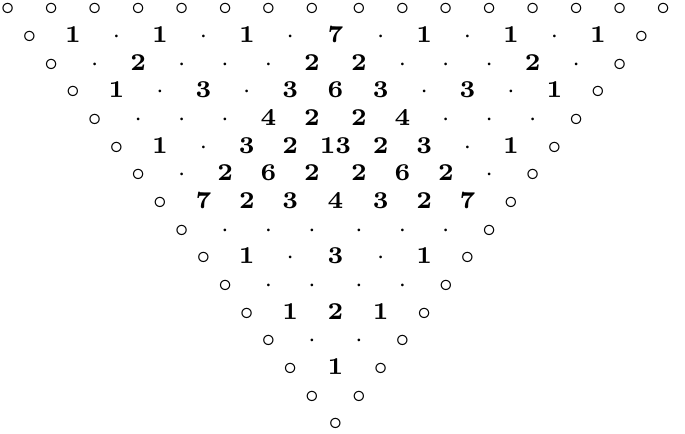}
\caption{\ $J=15$: {\bf a} All zeros  and {\bf b} trivial zeros.}
\label{Figure 6}
\end{figure}

\[
\begin{array}{r@{}l@{\qquad}l}
Orbits: &\quad  \left|\left|\begin{array}{ccc}
\,  5 & \ 5\ &  5\ \\
\, 5 &  5 &  5 \ \\
\,  5   & \ 5\ & 5\ \\
\end{array}\right|\right|,\  \left|\left|\begin{array}{ccc}
  13\, &  1\ & 1\ \\
  1 &  7\  &  7\  \\
  1  &  7\ & 7\  \\
\end{array}\right|\right|, \left|\left|\begin{array}{ccc}
\, 9\ & 3\ &  3\ \\
\, 3\ &  6\ & 6\ \\
\, 3\  &  6\ &  6\  \\
\end{array}\right|\right|,  \left|\left|\begin{array}{ccc}
\, 1 & \ 7 \ & 7 \\
\, 7 & \ 4\  & 4 \\
 \, 7   & \ 4\ & 4 \\
\end{array}\right|\right|,\  \left|\left|\begin{array}{ccc}
 1 & \ 3 \, & 11 \\
\, 7 & \ 6\, & 2  \\
\, 7  & \ 6\, & 2 \\
\end{array}\right|\right|,\\[5mm]
 & \quad  \left|\left|\begin{array}{ccc}
\ 1 \ & 5\ &  9\ \\
 7 & 5\ &  3\  \\
  7   & 5\ &  3\  \\
\end{array}\right|\right|,\  \left|\left|\begin{array}{ccc}
\  3 & \ 5\ &  7\ \\
\ 6 & \ 5\  & 4\  \\
\   6 & \ 5\ & 4\  \\
\end{array}\right|\right|,  \left|\left|\begin{array}{ccc}
 \, 9 &\, 3 & 3 \\
 \, 3 &\, 11  &  1  \\
  \, 3  &\, 1 & 11  \\
\end{array}\right|\right|.
\end{array}
\]

\begin{figure}
\includegraphics[scale=.8]{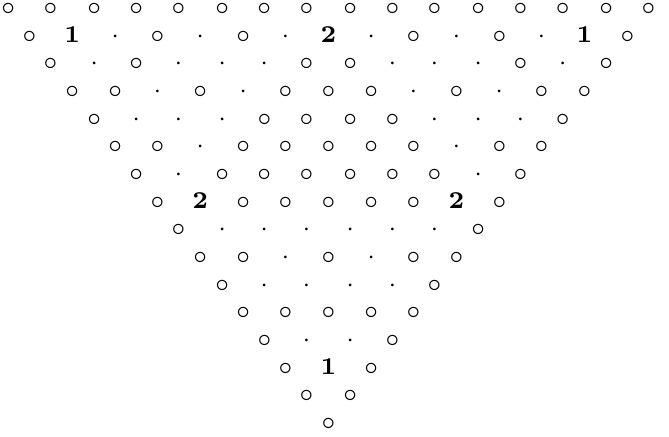}
\hspace{7pt}
\includegraphics[scale=.8]{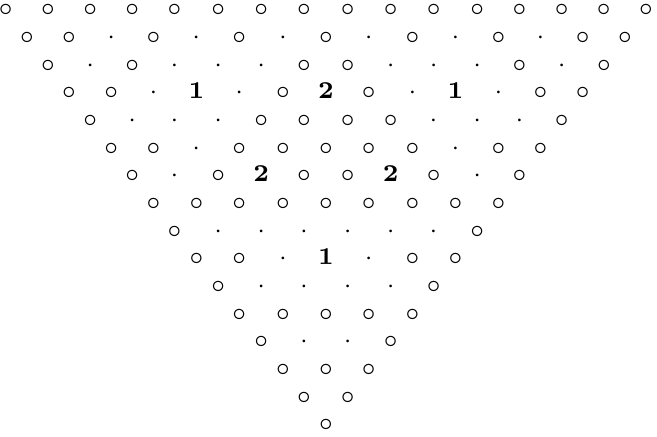}
\caption{\ $J=15$: Top lines {\bf a} $(13, 1, 1)$ and  {\bf b} $(9, 3, 3)$.}
\label{Figure 7}
\end{figure}

\begin{figure}
\includegraphics[scale=.8]{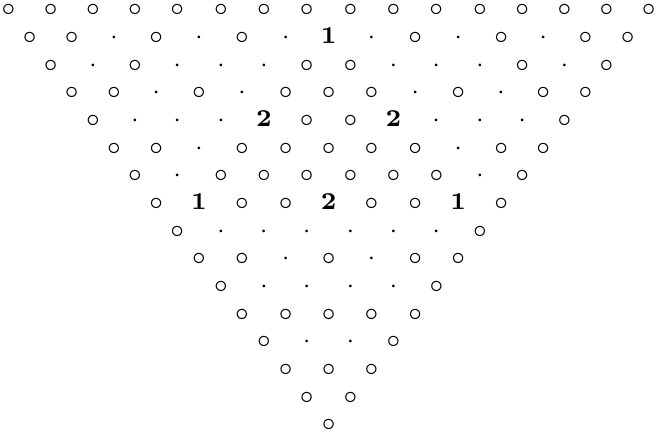}
\hspace{7pt}
\includegraphics[scale=.8]{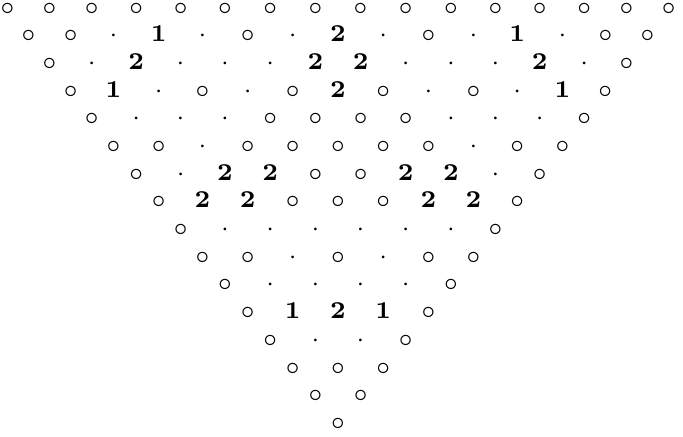}
\caption{\ $J=15$:  Top lines {\bf a} $(1, 7, 7)$ and  {\bf b} $(1, 3, 11)$.}
\label{Figure 8}
\end{figure}

\begin{figure}
\includegraphics[scale=.8]{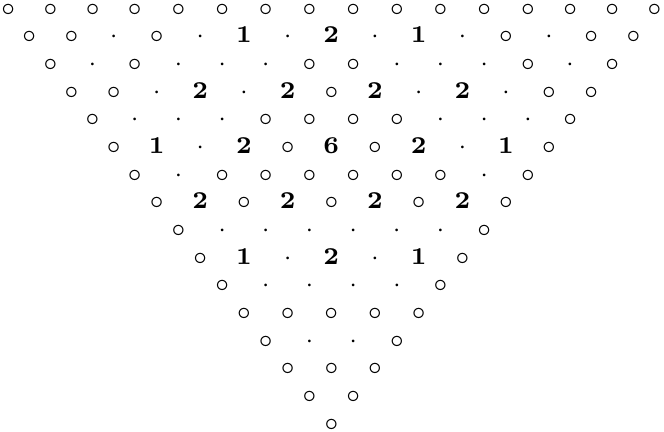}
\hspace{7pt}
\includegraphics[scale=.8]{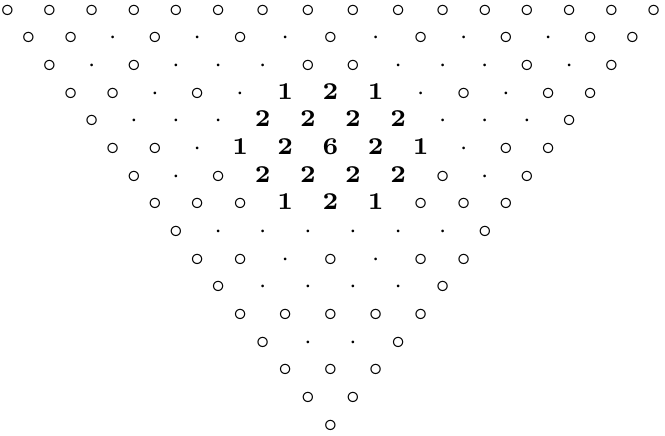}
\caption{\ $J=15$: Top lines {\bf a}  $(1, 5, 9)$  and  {\bf b}  $(3, 5, 7)$.}
\label{Figure 9}
\end{figure}

\begin{figure}
\includegraphics[scale=.8]{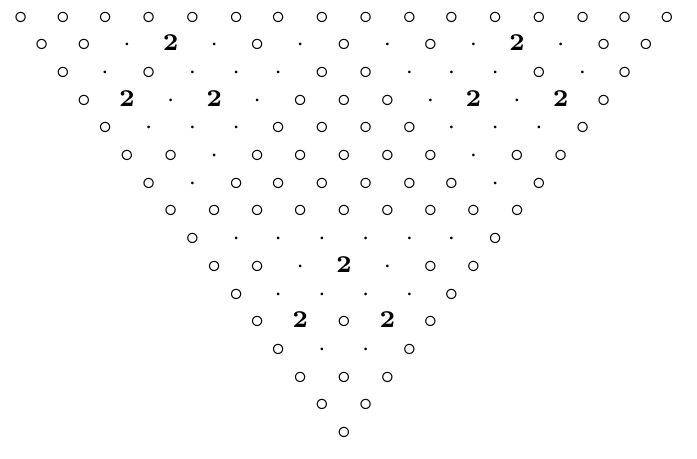}
\caption{\ $J=15$:  Non-trivial orbit with top line $(9, 3, 3)$.}
\label{Figure 10}
\end{figure}

\input{references}

\end{document}

%% file: references.tex
%%%%%%%%%%%%%%%%%%%%%%%% referenc.tex %%%%%%%%%%%%%%%%%%%%%%%%%%%%%%
% sample references
% %
% Use this file as a template for your own input.
%
%%%%%%%%%%%%%%%%%%%%%%%% Springer-Verlag %%%%%%%%%%%%%%%%%%%%%%%%%%
%
% BibTeX users please use
% \bibliographystyle{}
% \bibliography{}
%